\newfont{\bcb}{msbm10}
\newfont{\matb}{cmbx10}
\newfont{\got}{eufm10}

\documentclass[12pt]{amsart}
\usepackage{amsmath, amsthm, amscd, amsfonts, amssymb, latexsym, graphicx, color}
\usepackage[bookmarksnumbered, colorlinks, plainpages, hypertex]{hyperref}

\usepackage[cp1250]{inputenc}

\usepackage{amsmath,amsthm}
\usepackage{amssymb,latexsym}
\usepackage{enumerate}

\newtheorem{theorem}{Theorem}[section]

\newtheorem{proposition}[theorem]{Proposition}
\newtheorem{corollary}[theorem]{Corollary}
\theoremstyle{definition}

\theoremstyle{remark}
\newtheorem{remark}[theorem]{Remark}
\numberwithin{equation}{section}

\begin{document}

\title[On functions given by algebraic power series]{On functions given by algebraic power
       series over Henselian valued fields}

\author[Krzysztof Jan Nowak]{Krzysztof Jan Nowak}






\begin{abstract}
This paper provides, over Henselian valued fields, some theorems
on implicit function and of Artin--Mazur on algebraic power
series. Also discussed are certain versions of the theorems of
Abhyankar--Jung and Newton--Puiseux. The latter is used in
analysis of functions of one variable, definable in the language
of Denef--Pas, to obtain a theorem on existence of the limit,
proven over rank one valued fields in one of our recent papers.
This result along with the technique of fiber shrinking (developed
there over rank one valued fields) were, in turn, two basic tools
in the proof of the closedness theorem.
\end{abstract}

\maketitle

\section{Introduction}

Let $K$ be a Henselian valued field of equicharacteristic zero
with valuation $v$, valuation ring $R$ and value group $\Gamma$.
The main purpose of this paper is to examine algebraic power
series over $K$ and continuous functions given by them. To this
end, in Section~2 we state a version of the implicit function
theorem, and in the next section prove one of the Artin--Mazur
theorem on algebraic power series. Consequently, every algebraic
power series over $K$ determines a unique continuous function
which is definable in the language of valued fields. Section~4
presents certain versions of the theorems of Abhyankar--Jung and
Newton-Puiseux for Henselian subalgebras of formal power series
which are closed under power substitution and division by a
coordinate, given in our paper~\cite{Now1}.

\vspace{1ex}

The last section is devoted to functions of one variable definable
in the language of Denef--Pas. Having the Newton--Puiseux theorem
for algebraic power series at hand, we briefly outline how to
adapt the proof of the theorem on existence of the limit which was
proven over rank one valued fields in our
paper~\cite[Proposition~5.2]{Now2}. This result along with the
technique of fiber shrinking from~\cite[Section~6]{Now2} were, in
turn, two basic tools used in the proof of the closedness
theorem~\cite[Theorem~3.1]{Now2} over Henselian rank one valued
fields.

\vspace{1ex}

We should finally mention that the closedness theorem enables
i.a.\ application of resolution of singularities and
transformation to a normal crossing by blowing up in much the same
way as over the locally compact ground field. Some other its
applications in Henselian geometry are provided in our recent
papers~\cite{Now3,Now4}.

\section{Some versions of the implicit function theorem}

In this section, we give elementary proofs of some versions of the
inverse mapping and implicit function theorems (cf.~the versions
established in the papers \cite[Theorem~7.4]{P-Z} or
\cite[Proposition~3.1.4]{G-G-MB}). We begin with a simplest
version (H) of Hensel's lemma in several variables, studied by
Fisher~\cite{Fish}. Given an ideal $\mathfrak{m}$ of a ring $R$,
let $\mathfrak{m}^{\times n}$ stand for the $n$-fold Cartesian
product of $\mathfrak{m}$ and $R^{\times}$ for the set of units of
$R$. The origin $(0,\ldots,0) \in R^{n}$ is denoted by
$\mathbf{0}$.

\vspace{1ex}

\begin{em}
{\bf (H)} Assume that a ring $R$ satisfies Hensel's conditions
(i.e.\  it is linearly topologized, Hausdorff and complete) and
that an ideal $\mathfrak{m}$ of $R$ is closed. Let $f = (f_{1},
\ldots, f_{n})$ be an $n$-tuple of restricted power series $f_{1},
\ldots, f_{n} \in R\{ X \}$, $X = (X_{1},\ldots,X_{n})$, $J$ be
its Jacobian determinant and $a \in R^{n}$. If $f(\mathbf{0}) \in
\mathfrak{m}^{\times n}$ and $J(\mathbf{0}) \in R^{\times}$, then
there is a unique $a \in \mathfrak{m}^{\times n}$ such that $f(a)
= \mathbf{0}$.
\end{em}

\begin{proposition}\label{H-1}
Under the above assumptions, $f$ induces a bijection
$$ \mathfrak{m}^{\times n} \ni x \longrightarrow f(x) \in \mathfrak{m}^{\times n} $$
 of $\mathfrak{m}^{\times n}$ onto itself.
\end{proposition}

\begin{proof}
For any $y \in\mathfrak{m}^{\times n}$, apply condition (H) to the
restricted power series $f(X) - y$.
\end{proof}

If, moreover, the pair $(R,\mathfrak{m})$ satisfies Hensel's
conditions (i.e.\ every element of $\mathfrak{m}$ is topologically
nilpotent), then condition (H) holds by \cite[Chap.~III, \S
4.5]{Bour}.

\begin{remark}\label{rem-char}
Henselian local rings can be characterized both by the classical
Hensel lemma and by condition (H): a local ring $(R,\mathfrak{m})$
is Henselian iff $(R,\mathfrak{m})$ with the discrete topology
satisfies condition (H) (cf.~~\cite[Proposition~2]{Fish}).
\end{remark}

Now consider a Henselian local ring $(R,\mathfrak{m})$. Let $f =
(f_{1}, \ldots, f_{n})$ be an $n$-tuple of polynomials $f_{1},
\ldots, f_{n} \in R[ X ]$, $X = (X_{1},\ldots,X_{n})$ and $J$ be
its Jacobian determinant.

\begin{corollary}\label{H-2}
Suppose that $f(\mathbf{0}) \in \mathfrak{m}^{\times n}$ and
$J(\mathbf{0}) \in R^{\times}$. Then $f$ is a homeomorphism of
$\mathfrak{m}^{\times n}$ onto itself in the $\mathfrak{m}$-adic
topology. If, in addition, $R$ is a Henselian valued ring with
maximal ideal $\mathfrak{m}$, then $f$ is a homeomorphism of
$\mathfrak{m}^{\times n}$ onto itself in the valuation topology.
\end{corollary}

\begin{proof}
Obviously, $J(a) \in R^{\times}$ for every $a \in
\mathfrak{m}^{\times n}$. Let $\mathcal{M}$ be the jacobian matrix
of $f$. Then
$$ f(a + x) - f(a) = \mathcal{M}(a) \cdot x + g(x) = \mathcal{M}(a)
   \cdot (x + \mathcal{M}(a)^{-1} \cdot g(x)) $$
for an $n$-tuple $g = (g_{1},\ldots,g_{n})$ of polynomials
$g_{1},\ldots,g_{n} \in (X)^{2} R[X]$. Hence the assertion follows
easily.
\end{proof}

The proposition below is a version of the inverse mapping theorem.

\begin{proposition}\label{H-3}
If $f(\mathbf{0}) = \mathbf{0}$ and $e :=J(\mathbf{0}) \neq 0$,
then $f$ is an open embedding of $e^{2} \cdot \mathfrak{m}^{\times
n}$ into $e \cdot \mathfrak{m}^{\times n}$.
\end{proposition}

\begin{proof}
Let $\mathcal{N}$ be the adjugate of the matrix
$\mathcal{M}(\mathbf{0})$ and $y = e^{2}b$ with $b \in
\mathfrak{m}^{\times n}$. Since
$$ f(eX) = e \cdot \mathcal{M}(\mathbf{0}) \cdot X + e^{2} g(X) $$
for an $n$-tuple $g = (g_{1},\ldots,g_{n})$ of polynomials
$g_{1},\ldots,g_{n} \in (X)^{2} R[X]$, we get the equivalences
$$ f(eX) = y \ \Leftrightarrow \ f(eX) - y = \mathbf{0} \ \Leftrightarrow \
   e \cdot \mathcal{M}(\mathbf{0}) \cdot (X + \mathcal{N}g(X) - \mathcal{N}b) = \mathbf{0}. $$
Applying Corollary~\ref{H-2} to the map $h(X) := X +
\mathcal{N}g(X)$, we get
$$ f^{-1}(y) = ex \ \Leftrightarrow \ x = h^{-1}(\mathcal{N}b) \
   \ \text{and} \ \ f^{-1}(y) = e h^{-1}(\mathcal{N} \cdot y/e^{2}). $$
This finishes the proof.
\end{proof}

\vspace{1ex}

Further, let $0 \leq r < n$, $p = (p_{r+1}, \ldots, p_{n})$ be an
$(n-r)$-tuple of polynomials $p_{r+1},\ldots,p_{n} \in R[X]$, $X =
(X_{1},\ldots,X_{n})$,  and
$$ J := \frac{\partial(p_{r+1}, \ldots,
   p_{n})}{\partial(X_{r+1},\ldots,X_{n})}, \ \ e := J(\mathbf{0}). $$
Suppose that
$$ \mathbf{0} \in V := \{ x \in R^{n}: p_{r+1}(x) = \ldots = p_{n}(x) =
   0 \}. $$
In a similar fashion as above, we can establish the following
version of the implicit function theorem.

\begin{proposition}\label{implicit}
If $e \neq 0$, then there exists a continuous map
$$ \phi: (e^{2} \cdot \mathfrak{m})^{\times r} \longrightarrow
   (e \cdot \mathfrak{m})^{\times (n-r)} $$
which is definable in the language of valued fields and such that
$\phi(0)=0$ and the graph map
$$ (e^{2} \cdot \mathfrak{m})^{\times r} \ni u \longrightarrow (u,\phi(u)) \in
   (e^{2} \cdot \mathfrak{m})^{\times r} \times
   (e \cdot \mathfrak{m})^{\times (n-r)} $$
is an open embedding into the zero locus $V$ of the polynomials
$p$.
\end{proposition}

\begin{proof}
Put $f(X) := (X_{1},\ldots,X_{r},p(X))$; of course, the jacobian
determinant of $f$ at $\mathbf{0} \in R^{n}$ is equal to $e$. Keep
the notation from the proof of Proposition~\ref{H-3}, take any $b
\in e^{2} \cdot \mathfrak{m}^{\times r}$ and put $y := (e^{2}b,0)
\in R^{n}$. Then we have the equivalences
$$ f(eX)=y \ \Leftrightarrow \ f(eX) - y= \mathbf{0} \
   \Leftrightarrow \ e \mathcal{M}(\mathbf{0}) \cdot
   (X + \mathcal{N} g(X) - \mathcal{N} \cdot (b,0)) = \mathbf{0}. $$
Applying Corollary~\ref{H-2} to the map $h(X) := X +
\mathcal{N}g(X)$, we get
$$ f^{-1}(y) = ex \ \Leftrightarrow \ x = h^{-1}(\mathcal{N} \cdot (b,0)) \
   \ \text{and} \ \ f^{-1}(y) = e h^{-1}(\mathcal{N} \cdot y/e^{2}). $$
Therefore the function
$$ \phi(u) := e h^{-1}(\mathcal{N}\cdot (u,0)/e^{2}) $$
is the one we are looking for.
\end{proof}

\section{Density property and a version of the Artin--Mazur
theorem over Henselian valued fields}

We say that a topological field $K$ satisfies the {\it density
property} (cf.~\cite{K-N,Now2}) if the following equivalent
conditions hold.
\begin{enumerate}
\item If $X$ is a smooth, irreducible $K$-variety and
    $\emptyset\neq U\subset X$ is a Zariski open subset,
then $U(K)$ is dense in $X(K)$ in the $K$-topology.
\item If $C$ is a smooth, irreducible $K$-curve and
    $\emptyset\neq U$ is a Zariski open subset, then $U(K)$ is
    dense in
$C(K)$ in the $K$-topology.
\item If $C$ is a smooth, irreducible $K$-curve, then $C(K)$
    has no isolated points.
\end{enumerate}
(This property is indispensable for ensuring reasonable
topological and geometric properties of algebraic subsets of
$K^{n}$; see~\cite{Now2} for the case where the ground field $K$
is a Henselian rank one valued field.) The density property of
Henselian non-trivially valued fields follows immediately from
Proposition~\ref{implicit} and the Jacobian criterion for
smoothness (see e.g.~\cite[Theorem~16.19]{Eis}), recalled below
for the reader's convenience.

\begin{theorem}\label{smooth}
Let $I = (p_{1}, \ldots, p_{s}) \subset K[X]$, $X =
(X_{1},\ldots,X_{n})$ be an ideal, $A := K[X]/I$ and $V :=
\mathrm{Spec}\, (A)$. Suppose the origin $\mathbf{0} \in K^{n}$
lies in $V$ (equivalently, $I \subset (X)K[X]$) and $V$ is of
dimension $r$ at $\mathbf{0}$. Then the Jacobian matrix
$$ \mathcal{M} := \left[
   \frac{\partial p_{i}}{\partial X_{j}}(\mathbf{0}): \: i=1,\ldots,s, \: j=1,\ldots,n
   \right] $$
has rank $\leq (n-r)$ and $V$ is smooth at $\mathbf{0}$ iff
$\mathcal{M}$ has exactly rank $(n-r)$. Furthermore, if $V$ is
smooth at $\mathbf{0}$ and
$$ \mathcal{J} := \frac{\partial (p_{r+1},\ldots,p_{n})}{\partial (X_{r+1},\ldots,X_{n})}
   (\mathbf{0}) = \det \left[ \frac{\partial p_{i}}{\partial X_{j}} (\mathbf{0}):
   \: i,j=r+1,\ldots,n \right] \neq 0, $$
then $p_{r+1},\ldots,p_{n}$ generate the localization $I \cdot
K[X]_{(X_{1},\ldots,X_{n})}$ of the ideal $I$ with respect to the
maximal ideal $(X_{1},\ldots,X_{n})$.
\end{theorem}

\begin{remark}\label{etale}
Under the above assumptions, consider the completion
$$ \widehat{A} = K[[ X ]]/ I \cdot K[[ X ]] $$
of $A$ in the $(X)$-adic topology. If $\mathcal{J} \neq 0$, it
follows from the implicit function theorem for formal power series
that there are unique power series
$$ \phi_{r+1},\ldots,\phi_{n} \in (X_{1},\ldots,X_{r}) \cdot K[[X_{1},\ldots,X_{r}]] $$
such that
$$ p_{i}(X_{1},\ldots,X_{r},\phi_{r+1}(X_{1},\ldots,X_{r}), \ldots,
   \phi_{n}(X_{1},\ldots,X_{r})) = 0 $$
for $i=r+1,\ldots,n$. Therefore the homomorphism
$$ \widehat{\alpha}: \widehat{A} \longrightarrow K[[X_{1},\ldots,X_{r}]], \ \
   X_{j} \mapsto X_{j}, \ X_{k} \mapsto \phi_{k}(X_{1},\ldots,X_{r}), $$
for $j=1,\ldots,r$ and $k=r+1,\ldots,n$, is an isomorphism.

Conversely, suppose that $\widehat{\alpha}$ is an isomorphism;
this means that the projection from $V$ onto $\mathrm{Spec}\,
K[X_{1},\ldots,X_{r}]$ is etale at $\mathbf{0}$. Then the local
rings $A$ and $\widehat{A}$ are regular and, moreover, it is easy
to check that the determinant $\mathcal{J} \neq 0$ does not vanish
after perhaps renumbering the polynomials $p_{i}(X)$.
\end{remark}

We say that a formal power series $\phi \in K[[X]]$,
$X=(X_{1},\ldots,X_{n})$, is algebraic if it is algebraic over
$K[X]$. The kernel of the homomorphism of $K$-algebras
$$ \sigma: K[X,T] \longrightarrow K[[X]], \ \
   X_{1} \mapsto X_{1}, \, \ldots \, , X_{n} \mapsto X_{n}, \, T \mapsto \phi(X), $$
is, of course, a principal prime ideal:
$$ \mathrm{ker}\, \sigma = (p) \subset K[X,T], $$
where $p \in K[X,T]$ is a unique (up to a constant factor)
irreducible polynomial, called an \emph{irreducible polynomial} of
$\phi$.

\vspace{1ex}

We now state a version of the Artin--Mazur theorem
(cf.~\cite{AM,BCR} for the classical versions).

\begin{proposition}\label{A-M}
Let $\phi \in (X) K[[X]]$ be an algebraic formal power series.
Then there exist polynomials
$$ p_{1},\ldots ,p_{r} \in K[X,Y], \ \ Y=(Y_{1},\ldots,Y_{r}), $$
and formal power series $\phi_{2},\ldots,\phi_{r} \in K[[X]]$ such
that
$$ \mathcal{J} := \frac{\partial (p_{1},\ldots ,p_{r})}{\partial (Y_{1},\ldots ,Y_{r})}
   (\mathbf{0}) = \det \left[ \frac{\partial p_{i}}{\partial Y_{j}} (\mathbf{0}):
   \: i,j=1,\ldots,r \right] \neq 0, $$
and
$$ p_{i}(X_{1},\ldots,X_{n},\phi_{1}(X), \ldots, \phi_{r}(X)) = 0, \ \ i=1,\ldots,r, $$
where $\phi_{1} := \phi$.
\end{proposition}

\begin{proof}
Let $p_{1}(X,Y_{1})$ be an irreducible polynomial of $\phi_{1}$.
Then the integral closure $B$ of $A := K[X,Y_{1}]/(p_{1})$ is a
finite $A$-module and thus is of the form
$$ B = K[X,Y]/(p_{1},\ldots,p_{s}), \ \ Y=(Y_{1},\ldots,Y_{r}), $$
where $p_{1},\ldots,p_{s} \in K[X,Y]$. Obviously, $A$ and $B$ are
of dimension $n$, and the induced embedding $\alpha: A \to K[[X]]$
extends to an embedding $\beta: B \to K[[X]]$. Put
$$ \phi_{k} := \beta(Y_{k}) \in K[[X]], \ \ k=1,\ldots,r. $$
Substituting $Y_{k} - \phi_{k}(0)$ for $Y_{k}$, we may assume that
$\phi_{k}(0) = 0$ for all $k=1,\ldots,r$. Hence $p_{i}(\mathbf{0})
=0$ for all $i=1,\ldots,s$.

The completion $\widehat{B}$ of $B$ in the $(X,Y)$-adic topology
is a local ring of dimension $n$, and the induced homomorphism
$$ \widehat{\beta}: \widehat{B} =
   K[[X,Y]]/(p_{1},\ldots,p_{s}) \longrightarrow K[[X]] $$
is, of course, surjective. But, by the Zariski main theorem
(cf.~\cite[Chap.~VIII, \S~13, Theorem~32]{Z-S}), $\widehat{B}$ is
a normal domain. Comparison of dimensions shows that
$\widehat{\beta}$ is an isomorphism. Now, it follows from
Remark~\ref{etale} that the determinant $\mathcal{J} \neq 0$ does
not vanish after perhaps renumbering the polynomials $p_{i}(X)$.
This finishes the proof.
\end{proof}

Propositions~\ref{A-M} and~\ref{implicit} immediately yield the
following

\begin{corollary}
Let $\phi \in (X) K[[X]]$ be an algebraic power series with
irreducible polynomial $p(X,T) \in K[X,T]$. Then there is an $a
\in K$, $a \neq 0$, and a unique continuous function
$$ \widetilde{\phi}: a \cdot R^{n} \longrightarrow K $$
which is definable in the language of valued fields and such that
$\widetilde{\phi}(0) = 0$ and $p(x,\widetilde{\phi}(x)) =0$ for
all $x \in a \cdot R^{n}$. \hspace*{\fill}$\Box$
\end{corollary}

For simplicity, we shall denote the induced continuous function by
the same letter $\phi$. This abuse of notation will not lead to
confusion in general.

\begin{remark}
Clearly, the ring $K[[X]]_{alg}$ of algebraic power series is the
henselization of the local ring $K[X]_{(X)}$ of regular functions.
Therefore the implicit functions
$\phi_{r+1}(u),\ldots,\phi_{n}(u)$ from Proposition~\ref{implicit}
correspond to unique algebraic power series
$$ \phi_{r+1}(X_{1},\ldots,X_{r}),\ldots,\phi_{n}(X_{1},\ldots,X_{r}) $$
without constant term. In fact, one can deduce by means of the
classical version of the implicit function theorem for restricted
power series (cf.~\cite[Chap.~III, \S 4.5]{Bour} or~\cite{Fish})
that $\phi_{r+1},\ldots,\phi_{n}$ are of the form
$$ \phi_{k}(X_{1},\ldots,X_{r}) = e \cdot \omega_{k} (X_{1}/e^{2}, \ldots,
   X_{r}/e^{2}), \ \ k= r+1,\ldots, n, $$
where $\omega_{k}(X_{1},\ldots, X_{r}) \in
R[[X_{1},\ldots,X_{r}]]$ and $e \in R$.
\end{remark}

\section{The Newton--Puiseux and Abhyankar--Jung Theorems}

Here we are going to provide a version of the Newton--Puiseux
theorem, which will be used in analysis of definable functions of
one variable in the next section.

\vspace{1ex}

We call a polynomial
$$ f(X;T)= T^{s} + a_{s-1}(X)T^{n-1} + \cdots + a_{0}(X) \in K[[X]][T],
$$
$X=(X_{1},\ldots,X_{s})$, quasiordinary if its discriminant $D(X)$
is a normal crossing:
$$ D(X) = X^{\alpha} \cdot u(X) \ \ \ \mbox{ with } \ \ \alpha \in
   \mathbb{N}^{s}, \ u(X) \in k[[X]], \ u(0) \neq 0. $$

Let $K$ be an algebraically closed field of characteristic zero.
Consider a henselian $K[X]$-subalgebra $K \langle X \rangle$ of
the formal power series ring $K[[X]]$ which is closed under
reciprocal (whence it is a local ring), power substitution and
division by a coordinate. For positive integers
$r_{1},\ldots,r_{n}$ put
$$ K \langle X_{1}^{1/r_{1}},\ldots, X_{n}^{1/r_{n}} \rangle := \{
   a( X_{1}^{1/r_{1}},\ldots, X_{n}^{1/r_{n}}): a(X) \in K \langle X
   \rangle \}; $$
when $r_{1}= \ldots =r_{m}=r$, we denote the above algebra by $K
\langle X^{1/r} \rangle$.

\vspace{1ex}

In our paper~\cite{Now1}, we established a version of the
Abhyankar--Jung theorem.

\begin{proposition}\label{A-J}
Under the above assumptions, every quasi\-ordinary polynomial
$$ f(X;T)= T^{s} + a_{s-1}(X)T^{s-1} + \cdots + a_{0}(X)
   \in K \langle X \rangle[T] $$
has all its roots in $K \langle X^{1/r} \rangle$ for some $r \in
\mathbb{N}$; actually, one can take $r = s!$.
\end{proposition}

A particular case is the following version of the Newton-Puiseux
theorem.

\begin{corollary}\label{N-P}
Let $X$ denote one variable. Every polynomial
$$ f(X;T)= T^{s} + a_{s-1}(X)T^{s-1} + \cdots + a_{0}(X)
   \in K \langle X \rangle[T] $$
has all its roots in $K \langle X^{1/r} \rangle$ for some $r \in
\mathbb{N}$; one can take $r = s!$. Equivalently, the polynomial
$f(X^{r},T)$ splits into $T$-linear factors. If $f(X,T)$ is
irreducible, then $r=s$ will do and
$$ f(X^{s},T) = \prod_{i=1}^{s} \, (T - \phi(\epsilon^{i}X)), $$
where $\phi(X) \in K \langle X \rangle$ and $\epsilon$ is a
primitive root of unity.
\end{corollary}

\begin{remark}\label{puis}
Since the proof of these theorems is of finitary character, it is
easy to check that if the ground field $K$ of characteristic zero
is not algebraically closed, they remain valid for the Henselian
subalgebra $\overline{K} \otimes_{K} K \langle X \rangle$ of
$\overline{K}[[ X ]]$, where $\overline{K}$ denotes the algebraic
closure of $K$.
\end{remark}

\vspace{1ex}

The ring $K[[X]]_{alg}$ of algebraic power series is a local
Henselian ring closed under power substitutions and division by a
coordinate. Thus the above results apply to the algebra $K \langle
X \rangle = K[[X]]_{alg}$.

\section{Definable functions of one variable}

At this stage, we can readily to proceed with analysis of
definable functions of one variable over arbitrary Henselian
valued fields of equicharacteristic zero. We wish to obtain a
general version of the theorem on existence of the limit stated
below. It was proven in~\cite[Proposition~5.2]{Now2} over rank one
valued fields. Now the language $\mathcal{L}$ under consideration
is the three sorted language of Denef--Pas.

\begin{proposition}\label{limit-th} (Existence of the limit)
Let $f:A \to K$ be an $\mathcal{L}$-definable function on a subset
$A$ of $K$ and suppose $0$ is an accumulation point of $A$. Then
there is a finite partition of $A$ into $\mathcal{L}$-definable
sets $A_{1},\ldots,A_{r}$ and points $w_{1}\ldots,w_{r} \in
\mathbb{P}^{1}(K)$ such that
$$ \lim_{x \rightarrow 0}\, f|A_{j}\, (x) = w_{j} \ \ \ \text{for} \ \
   j=1,\ldots,r. $$
Moreover, there is a neighbourhood $U$ of $0$ such that each
definable set
$$ \{ (v(x), v(f(x))): \; x \in (A_{j} \cap U) \setminus \{0 \} \}
   \subset \Gamma \times (\Gamma \cup \ \{
   \infty \}),  \ j=1,\ldots,r, $$
is contained in an affine line with rational slope
$$ l = \frac{p_{j}}{q} \cdot k + \beta_{j}, \ j=1,\ldots,r, $$
with $p_{j},q \in \mathbb{Z}$, $q>0$, $\beta_{j} \in \Gamma$, or
in\/ $\Gamma \times \{ \infty \}$. \hspace*{\fill} $\Box$
\end{proposition}

\begin{proof}
Having the Newton--Puiseux theorem for algebraic power series at
hand, we can repeat mutatis mutandis the proof from loc.~cit.\ as
briefly outlined below. In that paper, the field $L$ is the
completion of the algebraic closure $\overline{K}$ of the ground
field $K$. Here, in view of Corollary~\ref{puis}, the $K$-algebras
$L\{ X \}$ and $\widehat{K}\{ X \}$ should be just replaced with
$\overline{K} \otimes_{K} K[[X]]_{alg}$ and $K[[X]]_{alg}$,
respectively. Then the reasonings follow almost verbatim. Note
also that Lemma~5.1 (to the effect that $K$ is a closed subspace
of $\overline{K}$) holds true for arbitrary Henselian valued
fields.
\end{proof}

We conclude with the following comment. The above proposition
along with the technique of fiber shrinking
from~\cite[Section~6]{Now2} were two basic tools in the proof of
the closedness theorem~\cite[Theorem~3.1]{Now2} over Henselian
rank one valued fields, which plays an important role in Henselian
geometry.

\vspace{2ex}

\begin{small}
Institute of Mathematics

Faculty of Mathematics and Computer Science

Jagiellonian University


ul.\ Profesora \L{}ojasiewicza 6, 30-348 Krak\'{o}w, Poland

{\em E-mail address: nowak@im.uj.edu.pl}
\end{small}

\end{document}